\documentclass[12pt,reqno ]{amsart}

\calclayout

\usepackage{amsmath}
\usepackage{amsrefs}
\usepackage{amsthm}
\usepackage{amssymb}
\usepackage{amsfonts}
\usepackage{mathrsfs}
\usepackage{comment}
\usepackage{hyperref}
\usepackage{enumerate,color}

\numberwithin{equation}{section}

\newtheorem{theorem}{Theorem}[section]
\newtheorem{lemma}[theorem]{Lemma}
\newtheorem{lem}[theorem]{Lemma}

\theoremstyle{definition}
\newtheorem*{remark}{Remark}

\def\Xint#1{\mathchoice
{\XXint\displaystyle\textstyle{#1}}%
{\XXint\textstyle\scriptstyle{#1}}%
{\XXint\scriptstyle\scriptscriptstyle{#1}}%
{\XXint\scriptscriptstyle\scriptscriptstyle{#1}}%
\!\int}
\def\XXint#1#2#3{{\setbox0=\hbox{$#1{#2#3}{\int}$}
\vcenter{\hbox{$#2#3$}}\kern-.5\wd0}}

\def\dashint{\Xint-}

\makeatletter
\newsavebox{\@brx}
\newcommand{\llangle}[1][]{\savebox{\@brx}{\(\m@th{#1\langle}\)}%
  \mathopen{\copy\@brx\kern-0.5\wd\@brx\usebox{\@brx}}}
\newcommand{\rrangle}[1][]{\savebox{\@brx}{\(\m@th{#1\rangle}\)}%
  \mathclose{\copy\@brx\kern-0.5\wd\@brx\usebox{\@brx}}}
\makeatother

\newcommand{\abs}[1]{\left|#1\right|}

\newcommand{\Cn}{\ensuremath{\mathbb{C}^n}}

\newcommand{\A}{\ensuremath{\text{{A}}}}
\newcommand{\R}{\ensuremath{\mathbb{R}}}

\newcommand{\Rd}{\ensuremath{\mathbb{R}^d}}
\newcommand{\D}{\ensuremath{\MC{D}}}

\newcommand{\MC}[1]{\ensuremath{\mathcal{#1}}}
\newcommand{\J}{\ensuremath{\mathscr{J}}}
\newcommand{\F}{\ensuremath{\mathscr{F}}}

\newcommand{\inrd}{\ensuremath{\int_{\Rd}}}

\newcommand{\innp}[1]{\left< #1 \right>}

\newcommand{\op}[1]{\ensuremath|#1|_{\text{op}}}

\newcommand{\Aonesc}[1]{\ensuremath{[#1]_{\text{A}_\infty ^\text{sc}}}}

\newcommand{\W}{\ensuremath{{\mathcal W}}}

\newcommand{\pr}[1]{\ensuremath{\left(#1\right)}}

\newcommand{\Ctwon}{\ensuremath{\mathbb{C}^{2n}}}

\DeclareMathOperator*{\esssup}{ess\,sup}
\allowdisplaybreaks

\def\Xint#1{\mathchoice
{\XXint\displaystyle\textstyle{#1}}%
{\XXint\textstyle\scriptstyle{#1}}%
{\XXint\scriptstyle\scriptscriptstyle{#1}}%
{\XXint\scriptscriptstyle\scriptscriptstyle{#1}}%
\!\int}
\def\XXint#1#2#3{{\setbox0=\hbox{$#1{#2#3}{\int}$}
\vcenter{\hbox{$#2#3$}}\kern-.5\wd0}}

\def\dashint{\Xint-}
\def\avgint{\Xint-}

\title{Weak endpoint bounds for matrix weights}

\author[D. Cruz-Uribe]{David Cruz-Uribe, OFS}
\address{Department of Mathematics\\
University of Alabama, Box 870350, 345 Gordon Palmer Hall.}\email{dcruzuribe@ua.edu}

\author[J. Isralowtiz]{Joshua Isralowitz}
\address{Department of Mathematics and Statistics\\
SUNY Albany, 1400 Washington Ave., Albany, NY, 12222.}\email{jisralowitz@albany.edu}

\author[K. Moen]{Kabe Moen}
\address{Department of Mathematics\\
University of Alabama, Box 870350, 345 Gordon Palmer Hall.}\email{kabe.moen@ua.edu}

\author[S. Pott]{Sandra Pott}
\address{Centre for Mathematical Sciences, University of Lund, PO Box 118, 22100 Lund, Sweden}\email{sandra@maths.lth.edu}

\author[I. Rivera-R\'{\i}os]{Israel P. Rivera-R\'{\i}os}
\address{Instituto de Matem\'atica de Bah\'{\i}a Blanca (INMABB), Departamento de Matem\'atica, Universidad Nacional del Sur (UNS) - CONICET, Av. Alem 1253, Bah\'{\i}a Blanca, Argentina}
\email{israel.rivera@uns.edu.ar}



\thanks{Cruz-Uribe is  supported by research funds from the Dean of
  the College of Arts \& Sciences, the University of Alabama;
  Isralowitz and Moen are
  supported by the Simons Foundation; {  Rivera-R\'{\i}os} is supported by grant  PIP (CONICET) 11220130100329CO}

\begin{document}

\maketitle

\begin{abstract}
  We prove quantitative matrix weighted endpoint estimates for the
  matrix weighted Hardy-Littlewood maximal operator,
  Calder\'on-Zygmund operators, and commutators
  of CZOs with scalar BMO functions,  when the matrix weight is in
  the class $\A_1$ introduced by M.~Frazier and S.~Roudenko.
\end{abstract}

\section{Introduction}

In this paper we consider weak-type endpoint estimates for operators
with matrix weights.  In order to put our results into context, we
first review briefly the scalar case.  It is well-known that the
Hardy-Littlewood maximal operator and all Calder\'{o}n-Zygmund
operators (CZOs)  are bounded on $L^p(w)$ when $p>1$ and $w\in A_p$:
that is,
$$[w]_{A_p}
=\sup_{Q}\left(\dashint_Q w(x)\,dx\right)\left(\dashint_Q
  w(x)^{1-p'}\,dx\right)^{p-1}<\infty,$$
where the supremum is taken over all cubes $Q$ with sides parallel to
the coordinate axes.  These operators are not bounded on $L^1(w)$, but
do map $L^1(w)$ into $L^{1,\infty}(w)$ when $w\in A_1$:  that is, for
every cube $Q$ and almost every $x\in Q$,
\[ \avgint_Q w(y)\,dy \leq [w]_{A_1} w(x); \]
here $[w]_{A_1}$ is the infimum of all constants such that this
inequality holds.

However, there is another version of the endpoint inequality.   Given
a weight $w$ and an operator $T$, for $1<p<\infty$, define  $T_wf= w^{-\frac{1}{p}}
T(w^{\frac{1}{p}}f)$.  Then it is immediate that strong-type
inequalities for   $T_w$ are equivalent to weighted strong-type
estimates for~$T$:   $T_w : L^p(\R^d) \rightarrow L^p(\R^d)$ if and
only if $T:L^p(w)\rightarrow L^p(w)$.   Consequently, we
have that if $T$ is a CZO or the Hardy-Littlewood maximal operator,
and if $w\in A_p$, then $T_w : L^p(\R^d) \rightarrow
L^{p,\infty}(\R^d)$.    This suggests that when $p=1$ and $w\in A_1$, we should have
weak-type inequalities of the form
\begin{equation} \label{eqn:MW1}
 |\{x:w(x) T(fw^{-1})(x)>\lambda\}|\lesssim
  \frac{1}{\lambda}\int_{\R^d}|f(x)| \, dx.
\end{equation}
Muckenhoupt and Wheeden~\cite{MW} first proved such inequalities when
$d=1$ for the Hardy-Littlewood maximal operator and the Hilbert
transform; their results were extended to higher dimensions and
arbitrary CZOs (as well as the maximal operator) in~\cite{CMP}.
These estimates are much more delicate and even for the maximal
operator are much more difficult to prove than the  more standard
endpoint result considered above.  Moreover,  it was shown in~\cite{MW}
that the $A_1$ condition is not necessary even for the maximal
operator; they showed, for instance, that this endpoint inequality holds when
$w(x)=|x|^{-1}$.

\begin{remark}
There has been a great deal of interest in generalizing these results to the
two-weight setting; the original motivation (even in the one-weight
setting) is that such inequalities
arise naturally in the {  theory of interpolation} with change of measure of
Stein and Weiss~\cite{stein-weiss58}.    See~\cite{CR} for a short history.
\end{remark}

We now consider matrix weights:  our goal is to
generalize~\eqref{eqn:MW1} to this setting.  To state our results we
first give some basic definitions.  For more details,
see~\cites{CRM,G,MR1928089}.  A matrix weight $W$ is an $n\times n$
self-adjoint matrix function with locally integrable entries such that
$W(x)$ is positive definite for a.e. $x\in \R^d$.
Define $W^r$ for any $r\in \R$ via diagonalization.  Define the operator norm of $W(x)$ by
\[ \op{W(x)} = \sup_{\substack{{e}\in \Cn\\|{e}|=1}}|W(x){e}|.  \]
Finally, for all $1 \leq p< \infty$, define $L^p(W)$ to be the
collection of measurable, vector-valued functions
$f:\R^d \rightarrow \Cn$ such that
$$\|{f}\|_{L^p(W)}=\left(\int_{\R^d} |W(x)^{\frac1p} f(x)|^p\,dx\right)^{\frac{1}{p}}<\infty.$$
%

%

Given a linear operator $T$, $1\leq p<\infty$, and a matrix weight $W$,  define the matrix
operator $T_W$ by
\begin{equation} \label{eqn:TW}
  T_Wf(x)=W(x)^{\frac1p}T(W^{-\frac1p}f)(x);
\end{equation}
as before, $T_W :
L^p(\R^d,\Cn) \rightarrow L^p(\R^d,\Cn)$ if and only if $T:L^p(W)
\rightarrow L^p(W)$.    It was shown by Christ and
Goldberg~\cites{CG,G} that for $p>1$, if
$T$ is a CZO, then this inequality holds if and only if $W$ satisfies
the matrix $\A_p$ condition,
\[ [W]_{\A_p} = \sup_Q \avgint_Q \bigg( \avgint_Q
\op{W(x)^{\frac{1}{p}}W(y)^{-\frac{1}{p}}}^{p'}\,dy\bigg)^{\frac{p}{p'}}
\,dx<\infty. \]

{ Note that this formulation of the matrix $\A_p$ condition is due to
Roudenko~\cite{MR1928089}; see this paper or~\cite{CG,G} for the
earlier, equivalent definition in terms of norms.}  While the maximal
operator is not linear, they showed that a variant of the
maximal operator (now referred to as the Christ-Goldberg maximal operator,
\begin{equation} \label{eqn:MW}
M_W{f}(x)=\sup_{Q\ni x} \dashint_Q
|W(x)^{\frac{1}{p}}W^{-\frac{1}{p}}(y){f}(y)|\,dy,
\end{equation}
is bounded on $L^p(\R^d,\Cn)$ if and only if $W\in \A_p$, $1<p<\infty$.

In the setting of matrix weights, it is unclear how to define the
weak-type space $L^{p,\infty}(W)$; therefore, it is natural to use the
strong-type bounds to get unweighted, weak $(p,p)$ bounds for $T_W$ or
$M_W$.    Here, we argue directly to prove
weak $(1,1)$ estimates for $T_W$ when $T$ is a CZO, and for the Christ-Goldberg maximal
operator  $M_W$.  The appropriate weight class is matrix $\A_1$, first
defined by Frazier and Roudenko~\cite{FR}:   a matrix weight $W$ belongs to $\A_1$ if { 
\begin{equation}\label{eq:A1}
  [W]_{\A_1}= \esssup_{x \in \Rd} \sup_{Q \ni x} \dashint_Q
  \op{W(y)W^{-1}(x)}\,dy < \infty.
\end{equation}
}
Note that in the scalar case, when $n=1$, this definition reduces
immediately to the definition of scalar $A_1$.  If $W\in \A_1$, then
we have that $\op{W(\cdot)}$ is in scalar $A_1$; see~\cite[Lemma~4.4]{CRM}.  Therefore, we can define
the scalar $A_\infty$ constant of $W$ by
\begin{equation*}
  \Aonesc{W} = { \sup_{{e} \in \Cn} \left[\abs{W(\cdot) {e}}\right]_{\text{A}_\infty}.}
\end{equation*}
This constant was first  introduced in the $p = 2$ setting in
\cite{NPTV} and for $1 < p < \infty$ in \cite{CIM}.   In the theory of
scalar weights there are several equivalent definitions of $A_\infty$;
we will use the sharp, Fujii-Wilson definition.  The
precise definition does not directly matter as we will use this
condition indirectly; see \cite{HP,HPR}
for details.

We can now state our first two results.

\begin{theorem} \label{thm:MainMax}
 Define $M_W$ by \eqref{eqn:MW} with $p=1$.   Given $W\in \A_1$, then for all ${f}\in L^1(\R^d,\Cn)$ and
  $\lambda>0,$
  \begin{equation}\label{eqn:max}
    |\{x\in \R^d: M_W{f}(x)>\lambda\}|
    \lesssim \frac{[W]_{\A_1} \Aonesc{W} }{\lambda} \int_{\R^d}
    |{f}(x)|\,dx.
  \end{equation}
\end{theorem}

\begin{theorem} \label{thm:MainCZO}
  Define $T_W$ by \eqref{eqn:TW} with $p=1$.   Given $W\in \A_1$ then for all ${f}\in L^1(\R^d,\Cn)$ and
  $\lambda>0,$
  \begin{equation}\label{eqn:CZO}
    |\{x\in \R^d: |T_W{f}(x)|>\lambda\}|
    \lesssim \frac{[W]_{\A_1} \Aonesc{W}}{\lambda} \int_{\R^d}
    |{f}(x)|\,dx.
  \end{equation}
\end{theorem}

It follows from the definition that $\Aonesc{W}\lesssim [W]_{\A_1}$,
so  in both of these results we can estimate the constant by
$[W]_{\A_1}^2$.  While we are able to give a quantitative estimate in
terms of the $\A_1$ constant, we
do not believe that either result is sharp.  In the scalar case, for
the Hardy-Littlewood maximal operator it is well known (though not
explicitly in the literature)  that the
sharp constant is $[w]_{A_1}$.  For CZOs the sharp constant is $[w]_{A_1}(1+\log([w]_{A_1}))$:  see Lerner, Ombrosi
and P\'erez { \cite{LeOP}} for the upper bound {and \cite{LNO} for the lower bound.}

\begin{remark}
  When $1<p<\infty$, direct estimates for the best constant in the
  weak $(p,p)$ inequalities for these
  operators are not known.  From the strong $(p,p)$ inequalities, we
  have that an upper bound on the constant for $M_W$ is
  $[W]_{\A_p}^{\frac{p'}{p}}$ (see~\cite{IM}) and for $T_W$ is
  $[W]_{\A_p}^{1+\frac{1}{p-1}-\frac{1}{p}}$ (see~\cite{CIM}).  It is
  an open question whether our techniques can be used to prove better
  weak type estimates.
\end{remark}

\begin{remark}
  As we noted above, even in the scalar case the $\A_1$ condition is not necessary for $M_W$
  to satisfy the weak $(1,1)$ inequality.  However, $\A_1$ weights are
  characterized by a weak $(1,1)$ inequality for a closely related,
  ``auxiliary'' maximal operator $M_W^\prime$, that was introduced by
  Christ and Goldberg~\cites{CG,G} and which plays an important role
  in studying $M_W$.  See~\cite[Theorem~1.21]{CIM} where this is
  proved in a more general context.
\end{remark}

\medskip

Finally, we can also use our techniques to prove a quantitative,
weak-type estimate for commutators of CZOs.  Let $T$ be a CZO and let
$b\in BMO$.  Define the commutator $[T, b]f(x) =T(bf) (x) - b (x) Tf
(x)$, and define  the matrix weighted commutator $C_{b, W}(T)f  = W [T,
b] W^{-1}f $.  Even in the scalar case  commutators are
more singular and do not satisfy weak $(1,1)$ bounds.  Rather, the natural endpoint condition
involves an $L\log L$ estimate:  see P\'erez and
Pradolini~\cites{perez95b,MR1827073}.    Let $\Phi(t) = t \log(e +
t)$; then we have the following result.

\begin{theorem} \label{thm:MainComm}
  Given $W\in \A_1$, then for all ${f}\in L^1(\R^d,\Cn)$ and
  $\lambda>0$,
\begin{multline*}
  |\{x \in \Rd : |C_{b, W} (T) f (x)| > \lambda\}|  \\
  \lesssim \|b\|_{BMO} [W]_{\A_1} \max\{\log ([W]_{\A_1} + e ),  \Aonesc{W} \}^2
\int_{\R^d} \Phi \pr{\frac{|f(x)|}{\lambda}} \,dx.
\end{multline*}
\end{theorem}



The remainder of this paper is organized as follows.  In
Section~\ref{section:prelim} we provide some preliminary results about
matrix weights and about domination via sparse operators.   In
Section~\ref{section:proofs-main} we prove Theorems~\ref{thm:MainMax}
and~\ref{thm:MainCZO}.  Finally, in Section~\ref{section:proofs-comm}
we prove Theorem~\ref{thm:MainComm}.

Throughout, $d$ will denote the dimension of the underlying space
$\R^d$, and $n$ will be dimension of space $\Cn$ in which vector-valued
functions take their range.  All matrices will be $n\times n$
matrices.   If we write $A\lesssim B$, then there exists a constant
$c$ such that $A\leq cB$.  By $A\equiv B$ we mean $A\lesssim B$ and
$B\lesssim A$. The implicit constants might depend on $n$, $d$,
or the given CZO, {  but will not depend on the matrix weight $W$.}

\section{Preliminaries}
\label{section:prelim}

In this section we gather a few additional facts about matrix weights
and also give the results on sparse domination which are central to
our proofs.

First, as we noted above, if $W\in \A_1$, then
we have that $\op{W(\cdot)}$ is a scalar $A_1$
weight~\cite[Lemma~4.4]{CRM}.   Moreover, we in fact have that for any
$e\in \Cn$, $|W(\cdot)e| \in A_1$, and
\begin{equation} \label{AonePropA}
 [W]_{\A_1} = \sup_{e\in \Cn} [|W(\cdot)e|]_{A_1}.
\end{equation}

Central to estimating matrix weighted operators is the concept of a
reducing matrix.  These were first introduced in~\cite{CG,G} when $p>1$, and when $p = 1$
in~\cite{FR}.  Given a norm $\rho$ on $\Cn$, it is well known (see
\cite[Lemma 11.4]{NT} for a self contained and simple proof) that
there exists a positive definite $n \times n$ matrix $A$ such that for any $e \in \Cn$ we have
$|Ae| \approx \rho(e)$. We refer to this matrix as a reducing
matrix of $\rho$.  (Note that the matrix $A$ is not unique, but
this is not important in practice.)  In particular, given a matrix weight $W$ and
any measurable $Q \subseteq \Rd$ with $0 < |Q| <\infty$, we have that
$e \mapsto \dashint_Q | W(y) e| \, dy$ is a norm on $\Cn$.  Hereafter
we will denote by $\W_Q$  any reducing matrix for this norm, so that
$$|\W_Q e| \approx \dashint_Q | W(y) e| \, dy.$$

We can also define the matrix $\A_1$ condition in terms of reducing
matrices.     Given a cube $Q$  and $x\in Q$, we have that if
$\{e_j\}_j$ is the standard basis for $\Cn$, then
\begin{multline*}
\op{\mathcal{W}_QW^{-1}(x)}\approx
\sum_{j=1}^n|\mathcal{W}_QW^{-1}(x)e_j|\\ \approx
\sum_{j=1}^n\dashint_Q|W(y)W^{-1}(x)e_j|\,dy\approx\dashint_Q
\op{W(y)W^{-1}(x)} \,dy.
\end{multline*}
Hence,
$$[W]_{A_1}\approx \sup_Q \sup_{x\in Q} \op{\mathcal{W}_QW^{-1}(x)}.$$

\bigskip

To prove our results, we will show that we can reduce each weak-type
estimate to proving an analogous result for a so-called sparse
operator.  To define these operators first we recall the machinery of general dyadic grids as
defined in~\cite{LN}; we refer the reader there for complete
details. We will need the fact that
every cube in $\R^d$ can be approximated by a dyadic cube from one of
finitely many dyadic grids (see the corollary of \cite[Theorem
3.1]{LN}).

\begin{lemma} \label{lem:3grid}
  There exist dyadic grids $\D^1,\ldots,\D^{3^d}$ such that given any cube
  $Q$ there exists $1\leq \alpha \leq 3^d$ and $Q^\alpha\in \D^\alpha$
  such that $Q\subset Q^\alpha$ and $\ell(Q^\alpha)\leq 6\ell(Q)$.
\end{lemma}

Given $\eta\in(0,1)$ we say that $\mathcal{S}\subset\D$
is a $\eta$-sparse family if for every $Q\in\mathcal{S}$ there exists
a measurable subset $E_{Q}\subset Q$ such that
\begin{enumerate}
\item $\eta|Q|\leq|E_{Q}|.$
\item The sets $E_{Q}$ are pairwise disjoint.
\end{enumerate}
Further, given $\Lambda > 1$ we say that $\mathcal{S} \subset \D$
is a $\Lambda$ Carleson family if for every $Q \in \MC{S}$,
$$\sum_{P \in \MC{S}, P \subseteq Q} |P| \leq \Lambda |Q|.$$
Clearly every $\eta$-sparse family is $\eta^{-1}$ Carleson,
since
$$\sum_{P \in \MC{S}, P \subseteq Q} |P| \leq \eta^{-1} \sum_{P \in
  \MC{S}, P \subseteq Q} |E_P| \leq \Lambda^{-1} |Q|.$$
Though less obvious, the converse is true:. every $\Lambda$ Carleson
family is $\Lambda^{-1}$ sparse~\cite[Lemma 6.3]{LN}.  We will also
use without further comment the fact that every $\Lambda$ Carleson
family can be written as a union of $m$ Carleson families, each of
which is $1 + \frac{\Lambda-1}{m}$ Carleson~\cite[Lemma 6.6]{LN}.
Hereafter we will sometimes refer to a family as sparse or Carleson
without reference to $\eta$ or $\Lambda$  if the
specific values of these constants are unimportant.

To estimate CZOs applied to vector-valued functions, we will use the
convex body domination that was introduced by F.~Nazarov,
S.~Petermichl and A.~Volberg~ \cite{NPTV}.  Given
a cube $Q$ and a function $f\in L^{1}(Q,\Cn)$, define
\[
\llangle f\rrangle_{Q}=\left\{ \dashint_{Q}f\varphi
  dx\,: \varphi\in
  B_{L^{\infty}}(Q)\right\},
\]
where $B_{L^{\infty}(Q)}=\left\{ \phi\in L^{\infty}(Q,\R)\,:\,\|\phi\|_{L^{\infty}}\leq1\right\} $.
They proved that $\llangle f\rrangle_Q$
is a symmetric, convex and compact set in $\mathbb{C}^{n}$.

The following result was first proved in~\cite{NPTV};  we give it here
in the version found in
\cite[Corollary~2.3.18]{HNotes}.   To state it, recall that given a linear operator $T$, the grand-maximal operator
$M_{T}$, defined by A.~Lerner~\cite{Le}, is
\[
M_{T}f(x)=\sup_{Q\ni x}\sup_{y\in Q}|T(f\chi_{\mathbb{R}^{n}\setminus3Q})(y)|.
\]

\begin{theorem}\label{Thm:ConvexBody}
  Let $T:L^{1}(\mathbb{R}^{d})\rightarrow L^{1,\infty}(\mathbb{R}^{d})$
be a linear operator such that  $M_{T}:L^{1}(\mathbb{R}^{d})\rightarrow L^{1,\infty}(\mathbb{R}^{d})$.
For {  $f\in L_{c}^{1}(\mathbb{R}^{d};\mathbb{C}^{n})$} and
$\varepsilon\in(0,1)$,
there exist  $3^d$,  $3^{-d}(1-\varepsilon)$-sparse collections of
dyadic cubes (drawn from the dyadic grids in Lemma~\ref{lem:3grid})
such that
\[
Tf(x)\in\frac{c_{d,n}c_{T}}{\varepsilon}\sum_{j=1}^{3^{d}}\sum_{Q\in\mathcal{S}_{j}}\llangle f\rrangle_{Q}\chi_{Q}(x)
\]
where $c_{T}=\|T\|_{L^{1}\rightarrow L^{1,\infty}}+\|M_{T}\|_{L^{1}\rightarrow L^{1,\infty}}$.
 In particular, there exist functions $k_{Q}\in B_{L^{\infty}(Q\times Q)}$
such that
\[
Tf(x)=\frac{c_{d,n}c_{T}}{\varepsilon}\sum_{j=1}^{3^{d}}{ \sum_{Q\in\mathcal{S}_{j}}}\left(\dashint_{Q} k_{Q}(x,y)f(y)dy\right)\chi_{Q}(x).
\]
\end{theorem}

Lerner proved that for Calder\'{o}n-Zygmund operators,
\[
\|M_{T}\|_{L^{1}\rightarrow L^{1,\infty}}\leq
c_{d}\left(\|T\|_{L^{2}\rightarrow
    L^{2}}+c_{K}+\|\omega\|_{\text{Dini}}\right).
\]
See~\cite{Le} for the precise definitions of the Dini condition and the Dini ``norm" of the kernel $K$ of a CZO $T$.
It is also known that
\[
\|T\|_{L^{1}\rightarrow L^{1,\infty}}\leq c_{d}\left(\|T\|_{L^{2}\rightarrow L^{2}}+\|\omega\|_{\text{Dini}}\right).
\]
Consequently Theorem \ref{Thm:ConvexBody}
holds for a Calder\'{o}n-Zygmund operator $T$ with
\[
c_{T}=\|T\|_{L^{2}\rightarrow L^{2}}+c_{K}+\|\omega\|_{\text{Dini}}.
\]

\medskip

The sparse convex body domination can also be extended to commutators.
In fact, somewhat surprisingly, we  can use a ``$2 \times 2$ block matrix trick"
inspired by \cite{GPTV} in conjunction with
Theorem~\ref{Thm:ConvexBody} to obtain the corresponding result
for commutators. In particular, the following was very recently proved
in \cite{IPR}, extending \cite[Theorem 1.1]{LORR} (and in fact
providing a very short proof of \cite[Theorem 1.1]{LORR}.)  For
completeness we include the relatively short proof.

\begin{theorem}
\label{Thm:ConvexBodyComm}Let $T:L^{1}(\mathbb{R}^{d})\rightarrow L^{1,\infty}(\mathbb{R}^{d})$
be a linear operator such that  $M_{T}:L^{1}(\mathbb{R}^{d})\rightarrow L^{1,\infty}(\mathbb{R}^{d})$.
For  {  $f\in L_{c}^{1}(\mathbb{R}^{d};\mathbb{C}^{n})$ and every $b\in L_{loc}^{1}(\Rd)$ where $bf \in L^1(\Rd, \Cn)$},
and $\varepsilon\in(0,1)$ there exist there exist  $3^d$,  $3^{-d}(1-\varepsilon)$-sparse collections of
dyadic cubes (drawn from the dyadic grids in Lemma~\ref{lem:3grid}) such that
\[
  [b,T]f(x)
  \in\frac{c_{d,n}c_{T}}{\varepsilon}\sum_{j=1}^{3^{d}}\sum_{Q\in\mathcal{S}_{j}}\left[(b(x)-b_{Q})\llangle
    f\rrangle_{Q}\chi_{Q}(x)
    + \llangle(b-b_{Q})f\rrangle_{Q}\chi_{Q}(x)\right]
\]
where each $C_{S}$ is a constant and $c_{T}=\|T\|_{L^{1}\rightarrow L^{1,\infty}}+\|M_{T}\|_{L^{1}\rightarrow L^{1,\infty}}$.
In particular, there exist functions $k_{Q} \in B_{L^{\infty}(Q\times Q)}$
such that

\begin{align}
  [b,T]f(x)
  &
    =\frac{c_{d,n}c_{T}}{\varepsilon}\sum_{j=1}^{3^{d}}\sum_{Q\in\mathcal{S}_{j}}(b(x)-b_{Q})
    \left(\dashint_{Q}k_{Q}(x,y)f(y)dy\right)\chi_{Q}(x) \label{EqCommOne}\\
  & \quad {  - }\frac{c_{d,n}c_{T}}{\varepsilon}\sum_{j=1}^{3^{d}}
    \sum_{Q\in\mathcal{S}_{j}}\left(\dashint_{Q}k_{Q}(x,y)(b(y)-b_{Q})f(y)dy\right)\chi_{Q}(x)
    \label{EqCommTwo}.
\end{align}

\begin{proof} Let $f \in L^\infty_c(\Rd,\Cn)$. By Theorem \ref{Thm:ConvexBody} there exists sparse collections of cubes $\{\MC{S}_j\}_{j=1}^{3^d}$ and  $k_Q(x, y)$ with $\|k_Q\|_{L^\infty(\Rd \times \Rd)} = 1$ where \begin{equation} Tf(x) = c_{d, n} c_T \sum_{j = 1}^{3^d} \sum_{Q \in \MC{S}_j} \pr{\dashint_Q k_{Q} (x, y) f(y) \, dy } \chi_Q(x). \label{Sparse} \end{equation}

  Define the $\Ctwon$ valued function $\tilde{f}$ by $$\tilde{f}(x) = \begin{pmatrix} f(x) \\ f(x) \end{pmatrix}$$ and define the $2n \times 2n$ block matrix $\Phi (x) $ by $$\Phi(x) = \begin{pmatrix}1_{n \times n} &{  b(x)} \otimes 1_{n \times n}   \\ 0 & 1_{n \times n}  \end{pmatrix} $$ so that $$\Phi^{-1} (x) = \begin{pmatrix}1_{n \times n} & - {  b(x)} \otimes 1_{n \times n}   \\ 0 & 1_{n \times n}  \end{pmatrix} $$

{  Direct}  computation shows $$\Phi(x) (T \Phi^{-1} \tilde{f}) (x) = \begin{pmatrix} Tf(x) - [T, b] f(x) \\ Tf (x) \end{pmatrix} $$  {  and}  $$\Phi^{-1} (y) \tilde{f}(y) = \begin{pmatrix}1_{n \times n} & - b(y) \otimes 1_{n \times n}   \\ 0 & 1_{n \times n}  \end{pmatrix} \begin{pmatrix} f(y) \\ f(y) \end{pmatrix} = \begin{pmatrix} f(y) - f(y) b(y)  \\ f(y) \end{pmatrix} $$
{  Since $\Phi^{-1} \tilde{f} \in L_c ^1$ , plugging $\Phi^{-1} \tilde{f}$ into \eqref{Sparse} gives}
 \begin{align*} \Phi(x) (T \Phi^{-1} \tilde{f}) (x) &  =  c_{d, n} c_T \sum_{j = 1}^{3^d} \sum_{Q \in \MC{S}_j} \Phi(x) \begin{pmatrix}  \langle  k_{Q} (x, \cdot) (f - fb) \rangle_Q \\ \langle k_{Q} (x, \cdot)  f \rangle_Q    \end{pmatrix} \chi_Q(x)
 \\ & = c_{d, n} c_T \sum_{j = 1}^{3^d} \sum_{Q \in \MC{S}_j}  \begin{pmatrix}  \langle  k_{Q} (x, \cdot) (f - fb) \rangle_Q + b(x)  \langle k_{Q} (x, \cdot)  f \rangle_Q \\ \langle k_{Q} (x, \cdot)  f \rangle_Q    \end{pmatrix} \chi_Q(x). \end{align*}

  However, adding and subtracting $\langle k_Q(x, \cdot) f \rangle_Q   \langle b \rangle_Q $  to the first component, we get
   \begin{align*} & \Phi(x)  (T \Phi^{-1} \tilde{f}) (x)
   \\ & =
c_{d, n} c_T \sum_{j = 1}^{3^d} \sum_{Q \in \MC{S}_j}  \begin{pmatrix}  \langle  k_{Q} (x, \cdot) (f - fb) \rangle_Q + b(x)  \langle k_{Q} (x, \cdot)  f \rangle_Q \\ \langle k_{Q} (x, \cdot)  f \rangle_Q    \end{pmatrix} \chi_Q(x)
   \\ & = c_{d, n} c_T \sum_{j = 1}^{3^d} \sum_{Q \in \MC{S}_j}  \begin{pmatrix}  \langle   k_{Q} (x, \cdot) f \rangle_Q + \langle k_{Q} (x, \cdot) f (\langle b\rangle_Q - b) \rangle_Q + (b(x) - \langle b\rangle_Q ) \langle k_{Q} (x, \cdot)  f \rangle_Q \\ \langle k_{Q} (x, \cdot)  f \rangle_Q    \end{pmatrix} \chi_Q(x).
   \end{align*}

   Thus, \begin{align*} [T, b] f(x) &  =   c_{d, n} c_T \sum_{j = 1}^{3^d} \sum_{Q \in \MC{S}_j}     \dashint_Q  k_{Q} (x, y)  (b(y) - \langle b\rangle_Q) f(y) \, dy
   \\ &  - (b(x) - \langle b\rangle_Q ) \pr{\dashint_Q f(y) k_{Q} (x, y)  f (y) \, dy. }\end{align*}

\end{proof}

\end{theorem}

\section{Proofs of Theorems \ref{thm:MainMax} and \ref{thm:MainCZO}}
\label{section:proofs-main}

We will first need to prove a sparse domination of the maximal function which is more useful for us than the linearization in \cite{CG,G}. Note that a similar stopping time argument was used in \cite{IPR} to prove the sharp bound $$\|M_W\|_{L^p \rightarrow L^p} \lesssim [W]_{\A_1} ^\frac{1}{p} $$ when $1 < p < \infty.$ Given a dyadic grid $\D$, define \begin{equation*} M_{ W} ^\D {f}(x) = \sup_{\substack{Q \ni x  \\ Q \in \D}} \dashint_Q |W(x) W^{-1} (y) {f}(y)| \, dy \end{equation*} and for a sparse collection $\MC{S}$ of dyadic cubes, let \begin{equation*}\MC{T}_{\MC{S}, W} {f} (x) = \sum_{Q \in \MC{S}}  \op{W(x)\MC{W}_Q ^{-1} }\innp{|\MC{W}_Q W^{-1} {f}|}_Q \chi_Q(x) \end{equation*}

\begin{lemma}\label{MaxSparse}Let $W$ be any matrix weight and $f \in L^1(\Rd;\Cn)$ have compact support.  Then there exists $3^d$ sparse collections $\MC{S}_j$ where \begin{equation*} M_{ W} {f}(x) \lesssim \sum_{j = 1}^{3^d} \MC{T}_{\MC{S}_j, W} {f} (x). \end{equation*} \end{lemma}

\begin{proof}    By Lemma \ref{lem:3grid}, it is enough to prove Lemma \ref{MaxSparse} for $M_W ^{\MC{D}} f $ where $\MC{D}$ is a fixed dyadic grid. Furthermore, since $f$ has compact support, we can replace $M_W ^{\MC{D}} f$ by $M_{W, J} ^{\MC{D}}$ for some $J \in \D$, where \begin{equation*} M_{ W, J} ^\D {f}(x) = { {\sup_{\substack{Q \ni x  \\ Q \in \D(J)}}}} \dashint_Q |W(x) W^{-1} (y) {f}(y)| \, dy. \end{equation*}  Let $\J(J)$ denote the maximal cubes $L \in \D(J)$ (if any exist) where \begin{equation*}\innp{|\MC{W}_J W^{-1} {f}|}_L > 2 \innp{|\MC{W}_J W^{-1} {f}|}_J. \end{equation*} By maximality we have \begin{align*} \sum_{L \in \J(J)} |L| & < \frac{1}{2 \innp{|\MC{W}_J W^{-1} {f}|}_J } \sum_{L \in \J(J)} \int_L |\MC{W}_J W^{-1}(y) {f}(y)| \, dy
\\ & \leq \frac{1}{2 \innp{|\MC{W}_J W^{-1} {f}|}_J } \int_J |\MC{W}_J W^{-1}(y) {f}(y)| \, dy
\\ & = \frac{ |J|}{2}. \end{align*}

Now let $\MC{F}(J)$ be the collection of cubes in $\D(J)$ that are not a subset of any cube $I \in \J(J)$.  We then have \begin{align*} \sup_{\substack{Q \ni x  \\ Q \in \D(J)}} &  \dashint_Q |W(x) W^{-1} (y) {f}(y)| \, dy
\\ & \leq \sup_{\substack{Q \ni x  \\ Q \in \F(J)}} \dashint_Q |W(x) W^{-1} (y) {f}(y) \chi_{\cup \J(J)} (y)| \, dy
\\ & \qquad + \sup_{\substack{Q \ni x  \\ Q \in \F(J)}} \dashint_Q |W(x) W^{-1} (y) {f}(y) \chi_{J \backslash \cup \J(J)}(y) | \, dy
 \\ & \qquad + \sum_{L \in \J(J)} \sup_{\substack{Q \ni x  \\ Q \in \D(L)}}\dashint_Q |W(x) W^{-1} (y) {f}(y) | \, dy
\\ & := A_1(x) + A_2(x) + \sum_{L \in \J(J)} \sup_{\substack{Q \ni x  \\ Q \in \D(L)}}\dashint_Q |W(x) W^{-1} (y) {f}(y) | \, dy. \end{align*}

We first estimate $A_1(x)$.  Let $x \in Q \in \F(J)$  and assume also $x \in I \in \J(J)$. Then by definition of $\F(J)$ we must have $I \subsetneq Q \subseteq J$ so that \begin{align*} A_1(x)
 & \leq \sup_{I \in \J(J)} \sup_{J \supseteq Q \varsupsetneq I \ni x} \dashint_Q |W(x) W^{-1} (y) {f}(y)| \, dy
\\ & \leq \op{W(x) \W_J ^{-1}} \sup_{I \in \J(J)}  \sup_{J \supseteq Q \varsupsetneq I } \dashint_Q |\W_J  W^{-1} (y) {f}(y)| \, dy
\\ & \leq 2 \op{W(x) \W_J ^{-1}} \innp{|\MC{W}_J W^{-1} {f}|}_J \chi_J(x)\end{align*} where the last inequality follows from maximality.

Next, to estimate $A_2(x)$, let $x \not \in  \cup_{L \in \J(J)} L$ and pick a sequence $L_k ^x$ of nested dyadic cubes where $$\{L_k ^x \} = \{L \in \F(J)  : x \in L\} = \{L \in \D(J) : x \in L\}.$$  But if $$\sup_k \innp{ |\MC{W}_J W^{-1} {f} | }_{L_k ^x} > 2 \innp{|\MC{W}_J W^{-1} {f}|}_J$$ then for some $k$ we have $$\innp{ |\MC{W}_J W^{-1} {f} | }_{L_k ^x} > 2\innp{|\MC{W}_J W^{-1} {f}|}_J$$ which means that $x \in L_k ^x \subseteq Q$ for some $Q \in \J(J)$.  Thus,
\begin{align*} A_2(x) & \leq \sup_{k} \dashint_{L_k ^x} |W(x) W^{-1} (y) {f}(y) | \, dy
\\ & \leq \op{W(x) \W_J ^{-1}} \sup_{k} \innp{ |\MC{W}_J W^{-1} {f}) |}_{L_k ^x}
\\ & \leq 2 \op{W(x) \W_J ^{-1}} \innp{|\MC{W}_J W^{-1} {f}|}_J \chi_J(x). \end{align*}  Iteration now completes the proof \end{proof}

\begin{proof}[Proof of Theorems~\ref{thm:MainMax} and~\ref{thm:MainCZO}]
 By Fatou's lemma for weak type estimates, we may assume that $f$ has compact support.    By Theorem \ref{Thm:ConvexBody} we have that there exists
\begin{align*}
|T_Wf(x)|&\lesssim \sum_{j=1}^{3^d}\sum_{Q\in \mathcal S_j} \dashint_Q|W(x)W^{-1}(y)f(y)|dy\cdot\chi_Q(x)\\
&\leq  \sum_{j=1}^{3^d}\sum_{Q\in \mathcal S_j} \op{W(x) \W_Q^{-1}}\dashint_Q\op{ \W_QW^{-1}(y)}|f(y)|dy\cdot\chi_Q(x)  \\
&\leq [W]_{\A_1} \sum_{j=1}^{3^d}\sum_{Q\in \mathcal S} \op{W(x) \W_Q^{-1}}\dashint_Q|f(y)|dy\cdot\chi_Q(x).
\end{align*}  Also note that Lemma \ref{MaxSparse} gives us the same estimate for $M_W f$.  Therefore, to prove Theorem \ref{thm:MainMax} and Theorem \ref{thm:MainCZO} it is enough to prove that
$$|\{x:|\MC{T} |f|(x)|>2\}|\lesssim \Aonesc{W} \|f\|_{L^1(\R^d;\Cn)}$$ where $\MC{T}$ is the \textit{scalar} linear operator defined by
$$\MC{T}h(x)=\sum_{Q\in \mathcal S} \op{W(x) \W_Q^{-1}}\dashint_Q h(y)dy\cdot\chi_Q(x)$$
and $\mathcal S$ is a sparse family of cubes contained in a dyadic grid $\D$.
 Let
$$\Omega=\{x:M^\D f(x)>2\}=\bigcup_j Q_j$$
where $\{Q_j\}$ are the maximal dyadic cubes that satisfy:
$$\dashint_{Q_j}|f(y)| \, dy > 2.$$
We use a Calder\'{o}n-Zygmund decomposition argument inspired by the arguments in \cite{CMP}.  Let $|f|=g+b$ where $b=\sum_j b_j$ and $b_j = (|f|-\dashint_{Q_j}|f|) \chi_{Q_j} $.  Then $g$ is a non-negative function that satisfies
$$g(x)\leq 2^{d+1} \qquad \int_{\R^d} g(x)\,dx\leq \int_{\R^d}|f(x)|\,dx$$
while each $b_j$ is supported on $Q_j$ and satisfies $\int_{Q_j}b_j(x)\,dx=0$.
Then we have
\begin{multline*}
|\{x:|\MC{T} |f|(x)|>2\}|\leq |\{x:\MC{T}g(x)>1\}|\\+|\{x\in \Omega: \MC{T}b(x)>1\}|+|\{x\notin \Omega: \MC{T}b(x)>1\}|.
\end{multline*}
Notice that the second term satisfies $|\Omega|\leq \|f\|_{L^1(\R^d;\Cn)}$.  Meanwhile, the third term is zero, since $\MC{T}b=0$.  Indeed, if $x\in Q$ and $Q \subseteq Q_j$ for some $j$ then obviously $x \in \Omega$. Thus,
$$\int_Q b(y)\,dy=\sum_{j:Q_j\subset Q}\int_{Q_j} b_j(y)\,dy=0,$$
so $\MC{T} b=0$.  Thus we are reduced to estimating $|\{x:\MC{T}g(x)>1\}|$.  By \eqref{AonePropA} and the sharp A${}_\infty$ reverse H\"{o}lder inequality (see \cite{HPR}), we can pick $c > 0$ independent of $W$ where if $q =  1 + \frac{c}{ \Aonesc{W}}$ then
\begin{align*}
\dashint_Q \op{W(x) \W_Q^{-1}}^q\,dx&\simeq \sum_{i=1}^n \dashint_Q { \abs{W(x) \W_Q^{-1}e_i}^q \,dx}
\\ & { \lesssim \sum_{i=1}^n \left(\dashint_Q \abs{W(x) \W_Q^{-1}e_i}\,dx\right)^q}  \simeq \sum_{i=1}^n |\mathcal W_Q\mathcal W_Q^{-1}e_i|^q\simeq 1.
\end{align*}
Let $r<q$ be such that $r'=2q'$. We will show that $\|\MC{T}\|_{L^r (\Rd) \rightarrow L^r (\Rd) } \lesssim \Aonesc{W}$ and thus
\begin{multline*}
|\{x:\MC{T} g(x)>1\}| \leq  \int_{\R^n}\MC{T}g(x)^r\,dx
  \lesssim  \Aonesc{W}^r  \int_{\R^d}g(x)^r\,dx  \lesssim \Aonesc{W} \int_{\R^d}|f(x)|\,dx
\end{multline*}
where in the last line we used that $\Aonesc{W}^r\leq \Aonesc{W}^{1+ \frac{c}{ \Aonesc{W}}}\lesssim \Aonesc{W}.$ To see that $\MC{T}$ is bounded on $L^r(\R^d)$ let $g$ and $h$ be non-negative scalar functions with $g\in L^r(\R^d)$ and $h\in L^{r'}(\R^d)$.  Then we use the well-known bound for the dyadic maximal function
$$\left(\inrd (M^\D g (x))^r \, dx\right)^{\frac1r} \leq r' \left( \inrd |g(x)| ^r \, dx\right)^{\frac1r}\\ \lesssim \Aonesc{W} \left( \inrd |g(x)| ^r \, dx\right)^{\frac1r}.$$
We have
\begin{align*}\lefteqn{\int_{\R^d}\MC{T}g(x)h(x)\,dx=\sum_{Q\in \mathcal S} \dashint_Q\op{W(x)\mathcal W_Q^{-1}}h(x)\,dx \dashint_Q g(y)\,dy \,|Q|}\\
&\leq \sum_{Q\in \mathcal S} \left(\dashint_Q\op{W(x)\mathcal W_Q^{-1}}^q\,dx\right)^{\frac1q}\left(\dashint_Q h(x)^{q'}\,dx\right)^{\frac{1}{q'}} \dashint_Q g(y)\,dy \,|Q|\\
&\lesssim \sum_{Q\in \mathcal S}\left(\dashint_Q h(x)^{q'}\,dx\right)^{\frac{1}{q'}} \dashint_Q g(y)\,dy \,|Q|\\
&\leq \left[\sum_{Q\in \mathcal S} \left(\dashint_Q h(x)^{q'}\,dx\right)^{\frac{r'}{q'}}|Q|\right]^{\frac1{r'}}\left[\sum_{Q\in \mathcal S} \left(\dashint_Q g(x)\,dx\right)^{r}|Q|\right]^{\frac1{r}} \\
&\lesssim \left[\sum_{Q\in \mathcal S} \left(\dashint_Q h(x)^{q'}\,dx\right)^{\frac{r'}{q'}}|E_Q|\right]^{\frac1{r'}}\left[\sum_{Q\in \mathcal S} \left(\dashint_Q g(x)\,dx\right)^{r}|E_Q|\right]^{\frac1{r}} \\
&\leq \left[\int_{\R^d} M^\D_{q'}h(x)^{r'}\,dx\right]^{\frac{1}{r'}}\left[\int_{\R^d} M^\D g(x)^{r}\,dx\right]^{\frac{1}{r}}\\
&\lesssim \Aonesc{W} \|h\|_{L^{r'}(\R^d)}\|g\|_{L^{r}(\R^d)}
\end{align*}
where we have additionally used that the maximal function $M^\D_{q'}h=M(|h|^{q'})^{\frac1{q'}}$ is bounded on $L^{r'}(\R^d)$ since  $q'<r'=2q'$ and
$$\|M^\D_{q'}\|_{L^{r'}(\R^d)\rightarrow L^{r'}(\R^d)}\leq \left[\left(\frac{r'}{q'}\right)'\right]^{\frac{1}{q'}}\leq C_d.$$
\end{proof}

\begin{remark}
  The proof of Theorem~\ref{thm:MainMax} using Lemma~\ref{MaxSparse}
  has the advantage that it allows us to simultaneously prove
  Theorem~\ref{thm:MainCZO}.    However, it is also possible to
  prove the endpoint estimate for $M_W$ ({ most likely with worse A${}_1$ and A${}_{\infty}^\text{sc}$ dependence}) more directly by modifying the
  original proof of the strong $(p,p)$ estimates due to
  Goldberg~\cite{G}.   This proof also relies on the
  Calder\'on-Zygmund decomposition into  ``good'' and ``bad''
  functions.  The estimate for the bad part is similar to the proof
  given above, while the estimate for the good part is more
  complicated and relies on the operator $N_Q$
  from~\cite{G}.  We leave the details of this proof to the interested
  reader.
\end{remark}

\section{Proof of Theorem \ref{thm:MainComm}}
\label{section:proofs-comm}

The Calder\'{o}n-Zygmund decomposition argument that was used to prove
Theorem \ref{thm:MainCZO} unfortunately does not work to handle the
sparse type operators in Theorem \ref{Thm:ConvexBodyComm} and instead
we need to employ ``slicing" arguments that are similar to the ones in~\cite{CR}.

First, we {  recall} a few  facts about the space $L \log L$ and $\exp L$
needed in the proof.  For details, see~\cite[Chapter~5]{MR2797562}.   Let
$\Phi(t) = t \log (e + t)$.  It is straightforward to show that $\Phi$
is submultiplicative:  for all $s, t > 0$, $\Phi(st) \lesssim \Phi(s)\Phi(t)$.
For a measurable, $\Cn$ valued function
$f$ and a measurable $Q \subset \Rd$ with $0 < |Q| < \infty$, define
the $L\log L$ norm by the Luxemburg norm
\begin{equation}
  \||f|\|_{L \log L, Q} =
  \inf\left\{\lambda > 0 : \dashint_Q \Phi\pr{\frac{|f(y)|}{\lambda}}
    \, dy \leq 1\right\}. \label{LuxNorm}
\end{equation}
The conjugate Young function of $\Phi$ is the function
$\bar{\Phi}(t)=e^t-1$.   We again define the  $\exp L$ norm by the
Luxemburg norm
\begin{equation*}
  \||f|\|_{\exp L, Q} = \inf\left\{\lambda > 0 :
    \dashint_Q { \bar{\Phi}} \pr{\frac{|f(y)|}{\lambda}} \, dy \leq 1\right\}.
\end{equation*}
Then we have the following H\"older inequality for these spaces:
\begin{equation}\label{eqn:OrlHold}
  \dashint_Q |h(x)g(x)|\,dx\lesssim \|h\|_{L\log L,Q}\|g\|_{\exp L,Q}.
\end{equation}
Finally we will need the exponential integrability of $BMO$ functions;
this is a consequence of  the classical John-Nirenberg theorem.  If
$b\in BMO$, then
\begin{equation}
  \|b-b_Q\|_{\exp L,Q}\leq c\|b\|_{BMO}.\label{eqn:JN}
\end{equation}

\medskip

We will { prove} the desired estimate in {  Theorem
\ref{thm:MainComm}} by first bounding the term \eqref{EqCommOne}.
This bound is given by the following lemma.

\begin{lem}
\label{Lem:EndpC1}Let  $\MC{S}$ be a sparse family and $$ \MC{T}_{\MC{S}, b, W} f (x) = \sum_{Q\in\mathcal{S}}|b(x)-b_{Q}|\left(\dashint_{Q}|W(x) W^{-1} (y) f(y)| \, dy\right)\chi_{Q}(x).$$ If $b\in BMO$ then
\[
\left|\{\mathcal{T}_{\mathcal{S},b,W} f(x)>\lambda\}\right|\lesssim \lambda^{-1} \|b\|_{BMO}[W]_{A_{1}} { [W]_{\text{A}_\infty  ^{\text{sc}}} }\max\left\{ \log\left([W]_{A_{1}}+e\right),\Aonesc{W}\right\} \|f\|_{L^1}.
\]
\end{lem}
\begin{proof}
Without loss of generality we may assume that $\lambda  = 1$ and $\|f\|_{L^1} = \|b\|_{BMO} = 1$. Furthermore, we may assume that $\MC{S}$ is $\frac45$ sparse.   If
\[
G=\{|\mathcal{T}_{\mathcal{S},b,W}f(x)|>1\}\setminus\{M(|f|)(x)>1\}.
\]
then
it suffices to prove that
\[
|G|\leq c_{d}[W]_{A_{1}} { [W]_{\text{A}_\infty  ^{\text{sc}}} }\max\left\{ \log\left([W]_{A_{1}}+e\right), { [W]_{\text{A}_\infty  ^{\text{sc}}} }\right\} +\frac{1}{2}|G|.
\]
Let $g=\chi_{G}$ and as before choose $s=1+\frac{c}{\Aonesc{W}}$ with $c$ independent of $W$ where $$ \pr{\dashint_Q \op{W(x) \W_Q ^{-1} } ^s \, dx} ^\frac{1}{s} \lesssim  \dashint_Q \op{W(x) \W_Q ^{-1}}  \, dx \approx \op{\W_Q \W_Q ^{-1}} = 1.  $$
We then have $$ |G| = \abs{\left\{ x \in G : \MC{T}_{\MC{S}, b, W} f(x)   > 1\right\}} \leq \int_G \MC{T}_{\MC{S}, b, W} f(x) \, dx $$ while

\begin{align*}\int_G & \MC{T}_{\MC{S}, b, W} f(x) \, dx
\\ & = \sum_{Q \in \MC{S}} \dashint_{Q}\int_{Q\cap G}|b(x)-b_{Q}|{ \abs{W(x)W^{-1}(y)f(y)}}dxdy
\\
 & \leq \sum_{Q \in \MC{S}} \dashint_{Q} g(x) |b(x)-b_{Q}|\op{W(x)\mathcal{W}_{Q}^{-1}}dx\int_{Q}|\mathcal{W}_{Q}W^{-1}(y)f(y)|dy
 \\& \leq \sum_{Q \in \MC{S}} \left(\dashint_{Q}\op{W(x)\mathcal{W}_{Q}^{-1}}^{s}dx\right)^{\frac{1}{s}}\left(\dashint_{Q}g\right)^{\frac{1}{2s'}}\left(\dashint_{Q}|b(x)-b_{Q}|^{2s'} \, dx \right)^{\frac{1}{2s'}}
 \\ & \qquad \times \int_{Q}
{ \abs{\mathcal W_{Q}W^{-1}(y)f(y)}}dy |Q|
\\ & \leq c_{d}[W]_{A_{1}}\Aonesc{W}\sum_{Q \in \MC{S}} \left(\dashint_{Q}g\right)^{\frac{1}{2s'}}\dashint_{Q}|f(y)|dy \, |Q|.
\end{align*}
{Here we have used the following corollary of John-Nirenberg inequality 
$$\left(\dashint_{Q}|b(x)-b_{Q}|^{2s'} \, dx \right)^{\frac{1}{2s'}}{  \lesssim} s'\|b\|_{BMO} \approx \Aonesc{W}$$
see for instance \cite[Corollary 3.10 p.166]{GCRdF}.}
Now we split the sparse family as follows: We say $Q\in\mathcal{S}_{k,j}$,
$k,j\geq0$ if
\[
\begin{split}2^{-j-1} & <\dashint_{Q}|f(y)|dy\leq2^{-j}, \\
2^{-k-1} & <\langle g\rangle_{Q}^\frac{1}{2s'} \leq2^{-k}
\end{split}
\]
Then

\begin{align*}\int_G  \MC{T}_{\MC{S}, b, W} f(x) \, dx &\leq c_{d}[W]_{A_{1}}\Aonesc{W}\sum_{j=0}^{\infty}\sum_{k=0}^{\infty}\sum_{Q\in\mathcal{S}_{k,j}}\left(\dashint_{Q}g\right)^{\frac{1}{2s'}}\dashint_{Q}
|f(y)|dy|Q|\\
 & =\sum_{k=0}^{\infty}\sum_{j=0}^{\infty}s_{k,j}.
\end{align*}
Now we observe that
\[
s_{k,j}\leq\min\left\{ c_{d} 2\cdot2^{-k}[W]_{A_{1}}\Aonesc{W}, \,  c_{d}2^{-j}2^{k(2s'-1)+2s'}[W]_{A_{1}}\Aonesc{W}|G|\right\} := \alpha_{k, j}.
\]
For the first estimate we argue as follows. Let $E_{Q}=Q\setminus\bigcup_{\substack{Q'\in\mathcal{S}_{j,k} \\ Q' \subsetneq Q}}Q'$.
Then
\begin{align*} \int_{Q}|f(y)|dy &= \int_{E_{Q}}|f(y)|dy+\int_{\bigcup_{Q'\subsetneq Q}}|f(y)|dy
\\ & \leq \int_{E_{Q}}|f(y)|dy+\sum_{Q'\subsetneq Q}\int_{Q'}|f(y)|dy. \end{align*}
For the second term,
\[
\begin{split}\sum_{Q'\subsetneq Q}\int_{Q'}|f(y)|dy & \leq2^{-j}\sum_{Q'\subsetneq Q,Q'\in\mathcal{S}_{k,j}}|Q'|\\
 & \leq2^{-j - 2}|Q|\\
 & \leq\frac12 \int_{Q}|f(y)| \, dy
\end{split}
\]
since $\MC{S}$ is $\frac45$ sparse and thus $\frac54$ Carleson.  Thus,
\[
\int_{Q}|f(y)|dy\leq2\int_{E_{Q}}|f(y)| \, dy.
\]
which means that
\[
\begin{split}s_{k,j} & \leq2\sum_{Q\in\mathcal{S}_{j,k}}\int_{E_{Q}}|f|\left(\dashint_{Q}g\right)^{\frac{1}{2s'}}\\
 & \leq { 2 \cdot} 2^{-k}\sum_{Q\in\mathcal{S}_{j,k}}\int_{E_{Q}}|f|\\
 & \leq2\cdot2^{-k}\int_{\mathbb{R}^{d}}|f|=2\cdot2^{-k}.
\end{split}
\]
For the second estimate of $s_{k, j}$, let $S_{j, k} ^*$ denote the maximal cubes in $S_{j, k}$.  Then
\begin{equation}
\begin{split}s_{k,j} & \leq2^{-j}2^{-k}\sum_{Q\in\mathcal{S}_{j,k}}|Q|\\
 & \leq2^{-j}2^{-k}\sum_{Q\in\mathcal{S}_{j,k}^{*}}\sum_{P\subseteq Q}|P|\\
 & \leq \frac54 2^{-j}2^{-k}\sum_{Q\in\mathcal{S}_{j,k}^{*}}|Q|\\
 & = \frac54 2^{-j}2^{-k}\left|\bigcup_{Q\in\mathcal{S}_{j,k}}Q\right|\\
 & =\frac54 2^{-j}2^{-k}\left|\left\{ x\in\mathbb{R}^{d}\,:\,Mg>2^{-2s'k-2s'}\right\} \right|\\
 & \leq c_{d}2^{-j}2^{k(2s'-1)+2s'}|G|.
\end{split}
\label{eq:G-1}
\end{equation}
Since $\MC{S}$ is $\frac54$ Carleson.  Putting all this together, we obtain $$|G| \leq \sum_{k = 0}^\infty \sum_{j = 0}^\infty \alpha_{k ,j}.$$
Pick some $\gamma > 0$ to be determined momentarily.  To finish the proof we will estimate the double sum
\begin{align*}
\sum_{k=0}^{\infty}\sum_{j=0}^{\infty}\alpha_{k,j} & =\sum_{j\geq\left\lceil \log_{2}\left([W]_{A_{1}}[W]_{A_{\infty}}\gamma\right)\right\rceil
  +\left\lceil k(2s'-1)+2s'\right\rceil +k}\alpha_{k,j}
  \\ & \qquad +\sum_{j<\left\lceil \log_{2}\left([W]_{A_{1}}[W]_{A_{\infty}}\gamma\right)\right\rceil +\left\lceil k(2s'-1)+2s'\right\rceil +k}\alpha_{k,j}.
\end{align*}

For the first term
\begin{align*} & \sum_{j\geq\left\lceil \log_{2}\left([W]_{A_{1}}[W]_{A_{\infty}}\gamma\right)\right\rceil +\left\lceil k(2s'-1)+2s'\right\rceil +k}\alpha_{k,j}
 \\ & \leq c_{d}[W]_{A_{1}}\Aonesc{W}|G|\sum_{k=0}^{\infty}2^{k(2s'-1)+2s'}\sum_{j\geq\left\lceil \log_{2}\left([W]_{A_{1}}\Aonesc{W}\gamma\right)\right\rceil +\left\lceil k(2s'-1)+2s'\right\rceil +k}2^{-j}
 \\ & =c_{d}[W]_{A_{1}}\Aonesc{W}|G|\sum_{k=0}^{\infty}2^{k(2s'-1)+2s'}2^{-\left\lceil \log_{2}\left([W]_{A_{1}}\Aonesc{W}\gamma\right)\right\rceil -\left\lceil k(2s'-1)+2s'\right\rceil -k}\\
 & =c_{d}[W]_{A_{1}}\Aonesc{W}|G|\sum_{k=0}^{\infty}2^{k(2s'-1)+2s'}2^{-\left\lceil \log_{2}\left([W]_{A_{1}}\Aonesc{W}\gamma\right)\right\rceil -\left\lceil k(2s'-1)+2s'\right\rceil -k}\\
 & \leq\frac{c_{d}[W]_{A_{1}}\Aonesc{W}}{[W]_{A_{1}}\Aonesc{W}\gamma}|G|\sum_{k=0}^{\infty}2^{-k}\leq\frac{2c_{d}}{\gamma}|G|.
\end{align*}
And it suffices to choose $\gamma=4c_{d}$. For the second term

\begin{align*} & \sum_{j<\left\lceil \log_{2}\left([W]_{A_{1}}\Aonesc{W}4c_{d}\right)\right\rceil +k}\alpha_{k,j}
\\ & \leq {  c_{d}} 2\cdot\sum_{k=0}^{\infty}\sum_{1\leq j<\left\lceil \log_{2}\left([W]_{A_{1}}\Aonesc{W}4c_{d}\right)\right\rceil +\left\lceil k(2s'-1)+2s'\right\rceil +k}2^{-k}[W]_{A_{1}}\Aonesc{W}
\\  & \leq c_{d}2\cdot\sum_{k=0}^{\infty}\left(\log_{2}\left([W]_{A_{1}}\Aonesc{W}4c_{d}\right)+2(k+1)s'\right)2^{-k}[W]_{A_{1}}\Aonesc{W}
\\ & \leq c_{d}[W]_{A_{1}}\Aonesc{W}\max\left\{ \log\left([W]_{A_{1}}+e\right), \, \Aonesc{W}\right\}.
\end{align*}

Combining all the preceding estimates we are done.
\end{proof}

We now finish the proof of Theorem \ref{thm:MainComm} by estimating \eqref{EqCommTwo}.  More precisely, standard bounds on $M_{\Phi}$ in conjunction with the following lemma will finish the proof of Theorem \ref{thm:MainComm}. To prove the following Lemma we will need the exponential integrability of $BMO$ functions and   \begin{lem}
Let $\mathcal{S}$ be a $\frac{2 \Phi(2)}{2\Phi(2) + 1}$ sparse family and let $$\MC{T}_{S, b, W} ^* f (x) = \sum_{Q \in \MC{S}}   \dashint_Q \pr{|b(y) - b_Q| \, |W(x) W^{-1}(y) f(y)| \, dy }\chi_Q(x).$$ If $b\in BMO$ with $\|b\|_{BMO}=1$ and
\[
G =\{|\mathcal{T}_{b,W}^{*}f(x)|>1\}\setminus\{M_{\Phi}\left(|f|\right)(x)>1\}
\]
then we have that
\[
\left|G\right|\leq c_{d}[W]_{A_{1}}\max\left\{ \log\left([W]_{A_{1}}+e\right), \, \Aonesc{W} \right\} ^{2}\int_{\mathbb{R}^{d}}\Phi\left(|f|\right)+\frac{1}{2}|G|.
\]
\end{lem}
\begin{proof}
 Since $b\in BMO$ with $\|b\|_{BMO}=1$, by \eqref{eqn:JN} there exists a constant $c_d$ such that
 $$\|b-b_Q\|_{\exp L,Q}\leq c_d.$$
 If we again denote $g=\chi_{G}$ and proceed as in the proof of Lemma \ref{Lem:EndpC1} then

\begin{align*} \lefteqn{ \sum_{Q \in \MC{S}}  \dashint_{Q}\int_{Q\cap G}|b(y)-b_{Q}|\abs{W(x)W^{-1}(y)f(y)}dxdy}\\
 & \leq \sum_{Q \in \MC{S}} \dashint_{Q}\op{W(x)\W_{Q}^{-1}} g(x) \, dx\dashint_{Q}|b(y)-b_{Q}| |\W_{Q}W^{-1}(y)f(y)|dy \, |Q|\\
 & \leq \sum_{Q \in \MC{S}} {\left(\dashint_{Q}\op{W(x)\W_{Q}^{-1}}^{s}dx\right)^{\frac{1}{s}}}\left(\dashint_{Q}g\right)^{\frac{1}{s'}}\dashint_{Q}|b(y)-b_{Q}|
|\W_{Q}W^{-1}(y)f(y)|dy|Q|
\\  & \leq \sum_{Q \in \MC{S}} {\left(\dashint_{Q}\op{W(x)\W_{Q}^{-1}}^{s}dx\right)^{\frac{1}{s}}}[W]_{A_{1}}\left(\dashint_{Q}g\right)^{\frac{1}{s'}}\dashint_{Q}|b(y)-b_{Q}||f(y)|dy|Q|\\ & \leq c_{d}[W]_{A_{1}} \sum_{Q \in \MC{S}}  {\dashint_{Q}\op{W(x)\W_{Q}^{-1}} dx}\left(\dashint_{Q}g\right)^{\frac{1}{s'}} \|b-b_Q\|_{\exp L,Q}\||f|\|_{L\log L,Q}|Q|\\
 & \leq c_{d}[W]_{A_{1}}  \sum_{Q \in \MC{S}}  \left(\dashint_{Q}g\right)^{\frac{1}{s'}}\||f|\|_{L\log L,Q}|Q|.
\end{align*}

As before we say $Q\in\mathcal{S}_{k,j}$ if

\[
\begin{split}2^{-j-1} & <\big\||f|\big\|_{L\log L,Q}\leq2^{-j},\\
2^{-k-1} & <\langle g\rangle_{Q} ^\frac{1}{s'} \leq2^{-k}.
\end{split}
\]
Then
\[
\begin{split} & \sum_{Q\in\mathcal{S}}\left(\dashint_{Q}g\right)^{\frac{1}{s'}}\big\||f|\big\|_{L\log L,Q}|Q|\\
 & =\sum_{j=0}^{\infty}\sum_{k=0}^{\infty}\sum_{Q\in\mathcal{S}_{k,j}}\left(\dashint_{Q}g\right)^{\frac{1}{s'}}\big\||f|\big\|_{L\log L,Q}|Q|\\
 & =\sum_{k=0}^{\infty}\sum_{j=0}^{\infty}s_{k,j}.
\end{split}
\]
Now we are going to prove that
\[
s_{k,j}\leq \min\left\{
c_{d}2^{-k}j\int_{\mathbb{R}^{d}}\Phi\left(|f(y)|\right)dy, \,
c_{d}2^{-j}2^{k(s'-1)+s'}|G|\right\} := \alpha_{k, j}.
\]
We start with the first estimate. Again let $\displaystyle E_{Q} = Q\setminus\bigcup_{\substack{Q'\in\mathcal{S}_{k,j} \\ Q' \subsetneq Q}} Q'$.
First we note that
\[
\int_{Q}\Phi\left(2^{j}|f(y)|\right)dy\leq2\int_{E_{Q}}\Phi\left(2^{j}|f(y)|\right)dy.
\]
Indeed, since $\mathcal{S}_{k,j}$ is $1+\frac{1}{2\Phi(2)}$-Carleson and $2^{-j-1} <\||f|\|_{L\log L,Q}\leq2^{-j}$ we have that
\[
\begin{split} & \int_{Q}\Phi\left(2^{j}|f(y)|\right)dy\\
 & \leq\int_{E_{Q}}\Phi\left(2^{j}|f(y)|\right)dy+\sum_{Q'\subsetneq Q,Q'\in\mathcal{S}_{k,j}}\int_{Q'}\Phi\left(2^{j}|f(y)|\right)dy.\\
 & \leq\int_{E_{Q}}\Phi\left(2^{j}|f(y)|\right)dy+\sum_{Q'\subsetneq Q,Q'\in\mathcal{S}_{k,j}}|Q'|\\
 & \leq\int_{E_{Q}}\Phi\left(2^{j}|f(y)|\right)dy+\frac{1}{2 \Phi(2)}\int_{Q}\Phi\left(2^{j+1}|f(y)|\right)dy\\
 & \leq\int_{E_{Q}}\Phi\left(2^{j}|f(y)|\right)dy+ \frac{1}{2}\int_{Q}\Phi\left(2^{j}|f(y)|\right)dy.
\end{split}
\]
Now we observe that
\[
\begin{split}s_{k,j} & \leq2^{-j}\sum_{Q\in\mathcal{S}_{j,k}}\int_{Q}\Phi\left(2^{j+1} |f(y)|\right)dy\left(\dashint_{Q}g\right)^{\frac{1}{s'}}\\
 & \leq c_{d}{ 2 \cdot }2^{-k}2^{-j} \Phi(2) \sum_{Q\in\mathcal{S}_{j,k}}\int_{E_{Q}}\Phi\left(2^{j} |f(y)|\right)dy\\
 & \leq c_{d}2^{-k}\Phi(2)j\int_{\mathbb{R}^{d}}\Phi\left(|f(y)|\right)dy
\end{split}
\]
{ while} the other estimate for $s_{j, k}$ follows from \eqref{eq:G-1}.  Combining the preceding estimates we have that
\[
\begin{split}|G| & \leq c_{d}[W]_{A_{1}}\sum_{k=0}^{\infty}\sum_{j=0}^{\infty}s_{k,j}\\
 & \leq\sum_{k=0}^{\infty}\sum_{j=0}^{\infty}\min\left\{ [W]_{A_{1}}c_{d}2^{-k}j\int_{\mathbb{R}^{d}}\Phi\left(|f(y)|\right)dy, \, c_{d}2^{-j}2^{k(s'-1)+s'}[W]_{A_{1}}|G|\right\} \\
 & =\sum_{k=0}^{\infty}\sum_{j=0}^{\infty}\alpha_{k,j}
\end{split}
\]
Now we are left with estimating the double sum. We proceed as follows.
\[
\sum_{k=0}^{\infty}\sum_{j=0}^{\infty}\alpha_{k,j}=\sum_{j\geq\left\lceil \log_{2}\left([W]_{A_{1}}\gamma\right)\right\rceil +\left\lceil k(s'-1)+s'\right\rceil +k}\alpha_{k,j}+\sum_{j<\left\lceil \log_{2}\left([W]_{A_{1}}\gamma\right)\right\rceil +\left\lceil k(s'-1)+s'\right\rceil +k}\alpha_{k,j}
\]
For the first term

\begin{align*}\lefteqn{\sum_{j\geq\left\lceil \log_{2}\left([W]_{A_{1}}\gamma\right)\right\rceil +\left\lceil k(s'-1)+s'\right\rceil +k}\alpha_{k,j} }\\
&\leq c_{d}[W]_{A_{1}}|G|\sum_{k=0}^{\infty}2^{k(s'-1)+s'}\sum_{j\geq\left\lceil \log_{2}\left([W]_{A_{1}}\gamma\right)\right\rceil +\left\lceil k(s'-1)+s'\right\rceil +k}2^{-j}\\
 & =c_{d}[W]_{A_{1}}|G|\sum_{k=0}^{\infty}2^{k(s'-1)+s'}2^{-\left\lceil \log_{2}\left([W]_{A_{1}}\gamma\right)\right\rceil -\left\lceil k(s'-1)+s'\right\rceil -k}\\
 & =c_{d}[W]_{A_{1}}|G|\sum_{k=0}^{\infty}2^{k(s'-1)+s'}2^{-\left\lceil \log_{2}\left([W]_{A_{1}}\gamma\right)\right\rceil -\left\lceil k(s'-1)+s'\right\rceil -k}\\
 & \leq\frac{c_{d}[W]_{A_{1}}}{[W]_{A_{1}}\gamma}|G|\sum_{k=0}^{\infty}2^{-k} \\ &
 { \leq \frac{2c_{d}}{\gamma}|G| = \frac12 |G|,}
\end{align*}

 {  where for the last equality we choose $\gamma=4c_{d}$.} For the second term

\begin{align*} \lefteqn{ \sum_{j<\left\lceil \log_{2}\left([W]_{A_{1}}\gamma\right)\right\rceil +\left\lceil k(s'-1)+s'\right\rceil +k}\alpha_{k,j}}\\
 & \leq c_{d}2[W]_{A_{1}}\int_{\mathbb{R}^{d}}\Phi\left(|f(y)|\right)dy\sum_{k=0}^{\infty}2^{-k}\sum_{1\leq j<\left\lceil \log_{2}\left([W]_{A_{1}}\gamma\right)\right\rceil +\left\lceil k(s'-1)+s'\right\rceil +k}j\\
 & \leq c_{d}2[W]_{A_{1}}\int_{\mathbb{R}^{d}}\Phi\left(|f(y)|\right)dy\sum_{k=0}^{\infty}2^{-k}\left(\left\lceil \log_{2}\left([W]_{A_{1}}\gamma\right)\right\rceil +\left\lceil k(s'-1)+s'\right\rceil +k\right)^{2}\\
 & \leq c_{d}2[W]_{A_{1}}\max\left\{ \log\left([W]_{A_{1}}+e\right), \, \Aonesc{W}\right\} ^{2}\int_{\mathbb{R}^{d}}\Phi\left(|f(y)|\right)dy.\\
\end{align*}

\end{proof}

\bibliography{Matrixbumps}

\end{document}